\documentclass[12pt,reqno]{amsart}
\usepackage{amsmath,amssymb,extarrows}
\usepackage{url}
\usepackage{tikz,enumerate}
\usepackage{diagbox}
\usepackage{appendix}
\usepackage{epic}

\usepackage{float} 

\usepackage{cite}

\usepackage{hyperref}
\usepackage{array}

\usepackage{booktabs}

\allowdisplaybreaks[4]

\newtheorem{theorem}{Theorem}
\newtheorem{corollary}[theorem]{Corollary}

\newtheorem{lemma}[theorem]{Lemma}

\theoremstyle{definition}

\theoremstyle{remark}

\numberwithin{equation}{section}
\numberwithin{theorem}{section}
\numberwithin{defn}{section}

\begin{document}
\title[On Some Double Nahm Sums of Zagier]
 {On Some Double Nahm Sums of Zagier}

\author{Zhineng Cao, Hjalmar Rosengren and Liuquan Wang}

\address{School of Mathematics and Statistics, Wuhan University, Wuhan 430072, Hubei, People's Republic of China}
\email{zhncao@whu.edu.cn}

\address{Mathematical Sciences, Chalmers University of Technology and University of Gothenburg, SE-41296 Gothenburg, Sweden}
\email{hjalmar@chalmers.se}

\address{School of Mathematics and Statistics, Wuhan University, Wuhan 430072, Hubei, People's Republic of China}
\email{wanglq@whu.edu.cn;mathlqwang@163.com}




\begin{abstract}
Zagier provided eleven conjectural rank two examples for Nahm's problem. All of them have been proved in the literature except for the fifth example, and there is no $q$-series proof for the tenth example. We prove that the fifth and the tenth examples are in fact equivalent. Then we give a $q$-series proof for the fifth example, which confirms a recent conjecture of Wang. This also serves as the first $q$-series proof for the tenth example, whose explicit form was conjectured by Vlasenko and Zwegers in 2011.
\end{abstract}

\maketitle

\section{Introduction and Main Results}
In his 1894 paper, Rogers \cite{Rogers1894} discovered many sum-product $q$-series identities. The most famous of these are
\begin{equation}
\sum_{n=0}^\infty \frac{q^{n^2}}{(q;q)_n}=\frac{1}{(q,q^4;q^5)_\infty}, \qquad
\sum_{n=0}^\infty \frac{q^{n(n+1)}}{(q;q)_n}=\frac{1}{(q^2,q^3;q^5)_\infty}. \label{RR}
\end{equation}
Here and in the whole paper we use standard $q$-series notation:
\begin{align}
(a;q)_0:=1, \quad (a;q)_n:=\prod\limits_{k=0}^{n-1}(1-aq^k), \quad (a;q)_\infty :=\prod\limits_{k=0}^\infty (1-aq^k), \\
(a_1,\cdots,a_m;q)_n:=(a_1;q)_n\cdots (a_m;q)_n, \quad n\in \mathbb{N}\cup \{\infty\}.
\end{align}
The identities \eqref{RR}  are called the Rogers--Ramanujan identities since they were rediscovered by Ramanujan some time before 1913. They have inspired a lot of work on finding identities of similar type. 
Rogers--Ramanujan-type identities serve as one of the witnesses for deep connections between the theory of $q$-series and modular forms. After multiplying with suitable powers of $q$, the right-hand sides of \eqref{RR} become modular forms. This fact is not clear from the sum sides. 

In general, it can be very hard to determine if a given $q$-hypergeometric series is modular or not. Motivated by applications in conformal field theory, Nahm 
 \cite{Nahm} suggested to study this question for a special class of multiple series. More precisely, 
given a positive integer $r$, Nahm's problem is to determine all $r\times r$ positive definite rational matrices $A$,  rational vectors $B$ of length $r$ and rational scalars $C$ such that
\begin{align}
f_{A,B,C}(q):=\sum_{n=(n_1,\dots,n_r)^\mathrm{T} \in (\mathbb{Z}_{\geq 0})^r} \frac{q^{\frac{1}{2}n^\mathrm{T} An+n^\mathrm{T} B+C}}{(q;q)_{n_1}\cdots (q;q)_{n_r}}
\end{align}
is a modular form. The series $f_{A,B,C}(q)$ is usually referred to as a Nahm sum, and $(A,B,C)$ is called a modular triple when $f_{A,B,C}(q)$ is modular.
 
When the rank $r=1$, Zagier \cite{Zagier} proved that there are exactly seven modular triples, including two that correspond to  the Rogers--Ramanujan identities \eqref{RR}. When the rank $r=2$, he made an extensive computer search that led to  eleven sets of possible modular triples recorded as Table 2 in \cite{Zagier}. 
All except one of these examples have been confirmed  by Zagier \cite{Zagier}, Vlasenko and Zwegers \cite{VZ}, 
Cherednik and Feigin \cite{Feigin},
Lee \cite{LeeThesis}, Calinescu, Milas and Penn \cite{CMP}, and Wang \cite{Wang-rank2}. A comprehensive discussion of all the eleven examples can be found in Wang's work \cite{Wang-rank2}. Following it, we label the examples in Zagier's list \cite[Table 2]{Zagier} from 1 to 11 according to their order of appearance.

Among these examples the fifth and the tenth are quite special. Until now, all of the eleven examples have been proved except for Example 5, which corresponds to
$$A=\begin{pmatrix} 1/3 & -1/3 \\ -1/3 & 4/3 \end{pmatrix}, \quad B\in \left\{
\begin{pmatrix} -1/6 \\ 2/3 \end{pmatrix},  \begin{pmatrix} 1/2 \\ 0 \end{pmatrix}\right\}.$$
Wang \cite{Wang-rank2} conjectured explicit modular product representations for the Nahm sums in this example, namely \cite[Conjecture 3.5]{Wang-rank2},
\begin{subequations}\label{exam5-wang}
\begin{align}
&S_1(q):=\sum_{i,j\geq 0}\frac{q^{\frac{1}{2}i^2-ij+2j^2-\frac{1}{2}i+2j}}{(q^3;q^3)_i(q^3;q^3)_j}  \nonumber \\
&=3\frac{(q^6;q^6)_\infty (q^{18},q^{27},q^{45};q^{45})_\infty}{(q^3;q^3)_\infty^2} -\frac{(q^4,q^6,q^{10};q^{10})_\infty(q,q^5,q^{11},q^{19},q^{25},q^{29};q^{30})_\infty}{(q^3;q^3)_\infty (q^3,q^{27};q^{30})_\infty},\label{exam5-1} \\
&S_2(q):=\sum_{i,j\geq 0}\frac{q^{\frac{1}{2}i^2-ij+2j^2+\frac{3}{2}i}}{(q^3;q^3)_i (q^3;q^3)_j} \nonumber \\
&=3q^2\frac{(q^6;q^6)_\infty (q^9,q^{36},q^{45};q^{45})_\infty}{(q^3;q^3)_\infty^2}
+\frac{(q^2,q^8,q^{10};q^{10})_\infty (q^5,q^7,q^{13},q^{17},q^{23},q^{25};q^{30})_\infty}{(q^3;q^3)_\infty (q^9,q^{21};q^{30})_\infty}. \label{exam5-2}
\end{align}
\end{subequations}

Example 10 is also special because until now there has been no $q$-series proof for it. For all other examples, the modularity was
confirmed by summing the series explicitly in terms of infinite products. 
By contrast, for Example 10, Cherednik and Feigin used nilpotent double affine Hecke algebras to 
prove the modularity of the sum side directly \cite[p.\ 1074]{Feigin}. 
The matrix and vector parts for this example are
$$A=\begin{pmatrix} 4/3 & 2/3 \\ 2/3 & 4/3 \end{pmatrix}, \quad
B\in \left\{\begin{pmatrix} -2/3 \\ -1/3 \end{pmatrix}, \begin{pmatrix}  -1/3 \\ -2/3 \end{pmatrix},  \begin{pmatrix} 0 \\ 0 \end{pmatrix}\right\}.$$
Since the $(1,1)$ and $(2,2)$ entries of $A$ are the same,  there are essentially only two Nahm sums involved. Vlasenko and Zwegers \cite{VZ} made the following conjecture for these Nahm sums:
\begin{subequations}
    \label{exam10-id}
\begin{align}
&S_3(q):=\sum_{i,j\geq 0} \frac{q^{2i^2+2ij+2j^2-2i-j}}{(q^3;q^3)_i(q^3;q^3)_j} \nonumber \\
&=\frac{(q^{45};q^{45})_\infty}{(q^3;q^3)_\infty} \left(2(q^{18},q^{27};q^{45})_\infty+q(q^{12},q^{33};q^{45})_\infty+q^4(q^3,q^{42};q^{45})_\infty\right), \label{exam10-1} \\
&S_4(q):=\sum_{i,j\geq 0} \frac{q^{2i^2+2ij+2j^2}}{(q^3;q^3)_i(q^3;q^3)_j} \nonumber \\
&=\frac{(q^{45};q^{45})_\infty}{(q^3;q^3)_\infty} \left((q^{21},q^{24};q^{45})_\infty -q^3(q^6,q^{39};q^{45})_\infty+2q^2(q^9,q^{36};q^{45})_\infty \right). \label{exam10-2}
\end{align}
\end{subequations}
These identities follow  from the work of
 Cherednik and Feigin, since knowing that both sides have the same modular behaviour it is straight-forward to verify that they agree.
 However, it remains to find direct proofs, using $q$-series techniques.

The main object of the present paper is to provide $q$-series proofs for Examples 5 and 10. Surprisingly, we will show  that these two examples are equivalent in the sense that they can be converted to each other. To make such conversions, we establish the following transformation formula.
\begin{theorem}\label{thm-general}
For $(a_k)_{k=-\infty}^\infty$ a general sequence such that the series below converge absolutely,
\begin{align}\label{eq-general}
\sum_{i,j\geq 0} \frac{a_{i-j}q^{\frac{j(j-1)}{2}}x^j}{(q;q)_i(q;q)_j}
=(-x;q)_\infty \sum_{i,j\geq 0} \frac{a_{i-j}q^{\frac{j(j-1)}{2}+ij} x^j}{(q;q)_i(q;q)_j(-x;q)_j}.
\end{align}
\end{theorem}
We will only use the special case when $a_{k}=q^{ak^2+bk}$ and $x=q$, where we need $a>0$ for convergence. It takes the form
\begin{equation}
\sum_{i,j\geq 0} \frac{q^{a(i-j)^2+b(i-j)+\frac{j(j+1)}{2}}}{(q;q)_i(q;q)_j}
=(-q;q)_\infty \sum_{i,j\geq 0} \frac{q^{a(i-j)^2+b(i-j)+\frac{j(j+1)}{2}+ij}}{(q;q)_i(q^2;q^2)_j}. \label{index-change}
\end{equation}
Using \eqref{index-change}
and the constant term method \cite[\S 4]{Andrews86}, we find that the Nahm sums in Examples 5 and 10 differ only by a simple factor.

\begin{theorem}\label{thm-exam-relation}
We have
\begin{equation}S_1(q)=(-q^3;q^3)_\infty S_3(q), \quad S_2(q)=(-q^3;q^3)_\infty S_4(q). \label{relation-1}\end{equation}
\end{theorem}

In light of this theorem, Examples 5 and 10 are equivalent and it suffices to prove only one of them. We will give a $q$-series proof for Example 5, which leads to the following result.

\begin{theorem}\label{thm-exam5}
We have
\begin{subequations}\label{exam5-id}
\begin{align}
S_1(q)&=\frac{(q^4;q^4)_\infty}{(q^3;q^3)_\infty (q^6;q^{12})_\infty}
\left(2\frac{(-q^4;q^4)_\infty^2}{(q^{24},q^{96};q^{120})_\infty} +q\frac{(-q^2;q^4)_\infty^2}{(q^{12},-q^{18};-q^{30})_\infty} \right), \label{exam5-id-1} \\
S_2(q)&=\frac{(q^4;q^4)_\infty}{(q^3;q^3)_\infty (q^6;q^{12})_\infty} \left(\frac{(-q^2;q^4)_\infty^2}{(-q^6,q^{24};-q^{30})_\infty} +2q^5\frac{(-q^4;q^4)_\infty^2}{(q^{48},q^{72};q^{120})_\infty}  \right). \label{exam5-id-2}
\end{align}
\end{subequations}
\end{theorem}

Note that the right-hand sides in \eqref{exam5-id} look different from the expressions in
\eqref{exam5-wang} and \eqref{exam10-id}. However, since all these expressions are modular 
(up to multiplication  by a power of $q$) the corresponding identities can be proved automatically. 
We have done this using the Maple package described in \cite{Garvan}. Thus, the following result is a consequence of Theorem \ref{thm-exam-relation} and Theorem \ref{thm-exam5}.

\begin{theorem}\label{thm-conj}
The conjectured identities \eqref{exam5-wang} and the previously known identities \eqref{exam10-id}   hold.
\end{theorem}

\section{Proofs of Theorems \ref{thm-general} and \ref{thm-exam-relation}}\label{sec-proof-relation}
We need Euler's $q$-exponential identities \cite[Corollary 2.2]{Andrews}
\begin{align}\label{Euler}
\sum_{n=0}^\infty \frac{z^n}{(q;q)_n}=\frac{1}{(z;q)_\infty}, \quad |z|<1,\qquad \sum_{n=0}^\infty \frac{q^{\binom{n}{2}} z^n}{(q;q)_n}=(-z;q)_\infty.
\end{align}
We introduce the multiplicative theta function
$$\theta(z;q):=(q,z,q/z;q)_\infty,$$
which satisfies the Jacobi triple product identity \cite[Theorem 2.8]{Andrews}
\begin{align}\label{Jacobi}
\theta(z;q)=\sum_{n=-\infty}^\infty (-1)^nq^{\binom{n}{2}}z^n
\end{align}
and the quasi-periodicity
\begin{equation}\label{theta-period}\theta(zq^k;q)=(-1)^kq^{-\frac{k(k-1)}2}z^{-k}\theta(z;q).\end{equation}
 For any series $f(z)=\sum_{n=-\infty}^\infty a(n)z^n$, we define the constant term extracting operator $\mathrm{CT}$ (with respect to $z$) as
$$\mathrm{CT} f(z)=a(0).$$

\begin{proof}[Proof of Theorem \ref{thm-general}]
We will use a special case of the $q$-Chu--Vandermonde summation formula,
written as in \cite[p.\ 20]{Andrews1984},
\begin{align}
\frac{1}{(q;q)_i(q;q)_j}=\sum_{k=0}^{\min(i,j)} \frac{q^{(i-k)(j-k)}}{(q;q)_k(q;q)_{i-k}(q;q)_{j-k}}. \label{Andrews-id}
\end{align}
It follows that
\begin{align*}
&\quad\sum_{i,j\geq 0} \frac{a_{i-j}q^{\frac{j(j-1)}{2}}x^j}{(q;q)_i(q;q)_j} 
=\sum_{i,j\geq 0} \sum_{k=0}^{\min(i,j)} \frac{a_{i-j}q^{\frac{j(j-1)}{2}+(i-k)(j-k)}x^j}{(q;q)_{i-k}(q;q)_{j-k}(q;q)_k}  \\
&=\sum_{i,j,k\geq 0} \frac{a_{i-j}q^{\frac{(j+k)(j+k-1)}{2}+ij}x^{j+k}}{(q;q)_i(q;q)_j(q;q)_k}
=\sum_{i,j\geq 0} \frac{a_{i-j}q^{\frac{j(j-1)}{2}+ij}x^{j}}{(q;q)_i(q;q)_j}\sum_{k=0}^\infty \frac{q^{\frac{k(k-1)}{2}+jk}x^{k}}{(q;q)_k}\\
&=(-x;q)_\infty \sum_{i,j\geq 0} \frac{a_{i-j}q^{\frac{j(j-1)}{2}+ij}x^{j}}{(q;q)_i(q,-x;q)_j},
\end{align*}
where we used \eqref{Euler} in the last step.
\end{proof}

Now we are able to prove the relations in Theorem \ref{thm-exam-relation}.
\begin{proof}[Proof of Theorem \ref{thm-exam-relation}]
It follows from \eqref{index-change} that
\begin{subequations}\label{S12-exp}
\begin{align}
S_1(q)&=\sum_{i,j\geq 0} \frac{q^{\frac{1}{2}i^2-ij+2j^2-\frac{1}{2}i+2j}}{(q^3;q^3)_i(q^3;q^3)_j}=(-q^3;q^3)_\infty \sum_{i,j\geq 0} \frac{q^{\frac{1}{2}i^2+2ij+2j^2-\frac{1}{2}i+2j}}{(q^3;q^3)_i(q^6;q^6)_j}, \label{S1-exp} \\
S_2(q)&=\sum_{i,j\geq 0} \frac{q^{\frac{1}{2}i^2-ij+2j^2+\frac{3}{2}i}}{(q^3;q^3)_i(q^3;q^3)_j} =(-q^3;q^3)_\infty \sum_{i,j\geq 0} \frac{q^{\frac{i^2}{2}+2ij+2j^2+\frac{3i}{2}}}{(q^3;q^3)_i(q^6;q^6)_j}. \label{S2-exp}
\end{align}
\end{subequations}
Next, we convert the double series in the right-hand sides to new forms using the constant term method. Consider the more general series
$$F(x,y):=\sum_{i,j\geq 0}\frac{q^{\frac 12(i+2j)(i+2j-1)}x^iy^j}{(q^3;q^3)_i(q^6;q^6)_j}.$$
By \eqref{Euler} and \eqref{Jacobi}, we can write
$$F(x,y)=\mathrm{CT}\sum_{i,j\geq 0}\sum_{n=-\infty}^\infty\frac{q^{\frac 12n(n-1)}x^iy^jz^{i+2j}(-1)^{i+n}}{(q^3;q^3)_i(q^6;q^6)_jz^n}
=\mathrm{CT}\frac{\theta(1/z;q)}{(-xz;q^3)(yz^2;q^6)_\infty}.
$$
In particular,
\begin{equation}\label{CT-S1} 
    \frac{S_1(q)}{(-q^3;q^3)_\infty }= F(1,q^3)=\mathrm{CT}\frac{\theta(1/z;q)}{(-z;q^3)_\infty(q^3z^2;q^6)_\infty}.
    \end{equation}
This can alternatively be written
\begin{align}\frac{S_1(q)}{(-q^3;q^3)_\infty }&=\mathrm{CT}\frac{(z;q^3)_\infty\theta(1/z;q)}{(z^2;q^3)_\infty}\nonumber\\
&=\mathrm{CT}\sum_{i,j\geq 0}\sum_{n=-\infty}^\infty\frac{(-1)^iq^{\frac 32i(i-1)}z^i}{(q^3;q^3)_i}\frac{z^{2j}}{(q^3;q^3)_j}\frac{(-1)^nq^{\frac{n(n-1)}2}}{z^n}\nonumber\\
&=\sum_{i,j\geq 0} \frac{q^{\frac{3}{2}i(i-1)+\frac{1}{2}(i+2j)(i+2j-1)}}{(q^3;q^3)_i (q^3;q^3)_j}\label{s1-alt},
\end{align}
which we recognize as the sum $S_3(q)$. 
In the same way,
\begin{align*}
\frac{S_2(q)}{(-q^3;q^3)_\infty }&=F(q^2,q)=\mathrm{CT}\frac{\theta(1/z;q)}{(-q^2z;q^3)_\infty(qz^2;q^6)_\infty}
=\mathrm{CT}\frac{(q^2z;q^3)_\infty\theta(1/z;q)}{(qz^2;q^3)_\infty}.
\end{align*}
Expanding this as in \eqref{s1-alt} we pick up an additional factor $q^{2i+j}$. That is, we arrive at the series $S_4(q)$. 
\end{proof}

\section{Proof of Theorem \ref{thm-exam5}}
The proof of Theorem \ref{thm-exam5} will be based on the four identities
\begin{subequations}
\label{four-rogers}
\begin{align}
\sum_{n=0}^\infty \frac{q^{n^2}}{(q;q)_{2n}}&=\frac{1}{(q;q^2)_\infty (q^4,q^{16};q^{20})_\infty},   \\
\sum_{n=0}^\infty \frac{q^{n(n+1)}}{(q;q)_{2n}}&=\frac{1}{(q;q^2)_\infty (-q,q^4;-q^5)_\infty},  \\
\sum_{n=0}^\infty \frac{q^{n(n+1)}}{(q;q)_{2n+1}}&=\frac{1}{(q;q^2)_\infty (q^2,-q^3;-q^5)_\infty},  \\
\sum_{n=0}^\infty \frac{q^{n(n+2)}}{(q;q)_{2n+1}}&=\frac{1}{(q;q^2)_\infty (q^8,q^{12};q^{20})_\infty}, 
\end{align}
\end{subequations}
which were obtained by Rogers in the same paper as \eqref{RR}  \cite[pp.\ 331--332]{Rogers1894}.

We observe that $S_1(q)=G(1,q^2)$ and $S_2(q)=G(q^{2},1)$, where
$$G(x,y):=\sum_{i,j=0}^\infty \frac{q^{\frac{i(i-1)}2-ij+2j^2}x^iy^j}{(q^3;q^3)_i(q^3;q^3)_j}. $$
 By the following result, these double series
 can both be reduced to single series.

\begin{lemma}
We have the reduction formula
\begin{multline}\label{ex5-reduction}(q^3;q^3)_\infty G(x,q^2/x)\\
=
\theta(-x;q^4)\sum_{j=0}^\infty\frac{q^{6j^2}x^{3j}}{(q^6;q^6)_{2j}}
+qx^2\theta(-q^{2}x;q^4)
\sum_{j=0}^\infty\frac{q^{6j(j+1)}x^{3j}}{(q^6;q^6)_{2j+1}}. \end{multline}
\end{lemma}

\begin{proof}
Since
$$ \frac{i^2}2-ij+2j^2=\frac{3j^2}2+\frac{(i-j)^2}2,$$
we can write
$$G(x,y)=\mathrm{CT}\sum_{i,j=0}^\infty\sum_{k=-\infty}^\infty
 \frac{q^{-\frac{i}2+\frac{3j^2}2+\frac{k^2}2}x^iy^jz^{i-j-k}}{(q^3;q^3)_i(q^3;q^3)_j}.
 $$
We multiply by $(-1)^{i+j+k}q^{(i-j-k)/2}$ before summing the series, which does not change the constant term. This gives
\begin{align*}G(x,y)&=\mathrm{CT} \sum_{i=0}^{\infty}
\frac{(-xz)^i}{(q^3;q^3)_i}
\sum_{j=0}^\infty
\frac{q^{\frac{3j(j-1)}2}(-qy/z)^j}{(q^3;q^3)_j}
\sum_{k=-\infty}^\infty
(-1)^kq^{\frac{k(k-1)}2}z^{-k}\\
&=\mathrm{CT}\frac{(qy/z;q^3)_\infty\theta(1/z;q)}{(-xz;q^3)_\infty},
\end{align*}
where we used \eqref{Euler} and \eqref{Jacobi}.
When $y=q^{2}/x$, this can be written
$$\mathrm{CT}\frac{(q^3/xz;q^3)_\infty\theta(1/z;q)}{(-xz;q^3)_\infty}=\mathrm{CT}\frac{\theta(q^3/xz;q^3)_\infty\theta(1/z;q)}{(q^3;q^3)_\infty(x^2z^2;q^6)_\infty}. $$
The numerator has the Laurent expansion
\begin{align*}\theta(q^3/xz;q^3)_\infty\theta(1/z;q)
&=\sum_{k=-\infty}^\infty (-1)^kq^{\frac{3k(k+1)}{2}}(xz)^{-k}
\sum_{l=-\infty}^\infty(-1)^lq^{\frac{l(l-1)}2}z^{-l}\\
&=\sum_{m=-\infty}^\infty(-1)^mz^{-m}q^{\frac{m(m-1)}2}\sum_{k=-\infty}^\infty q^{k(2k+2-m)}x^{-k}\\
&=\sum_{m=-\infty}^\infty(-1)^mz^{-m}q^{\frac{m(m-1)}2}
\theta(-q^{4-m}/x;q^4),
\end{align*}
where we substituted $l=m-k$ and again used \eqref{Jacobi}.
It follows that
\begin{align*}G(x,q^2/x)&=\frac 1{(q^3;q^3)_\infty}
\mathrm{CT}\sum_{j=0}^\infty\frac{(xz)^{2j}}{(q^6;q^6)_j}
\sum_{m=-\infty}^\infty(-1)^mz^{-m}q^{\frac{m(m-1)}2}
\theta(-q^{4-m}/x;q^4)\\
&=\frac 1{(q^3;q^3)_\infty}\sum_{j=0}^\infty\frac{q^{j(2j-1)}x^{2j}}{(q^6;q^6)_j}
\theta(-q^{4-2j}/x;q^4).\end{align*}
We split  the sum into terms with even and odd $j$. Using also  \eqref{theta-period}, the first term can be written
$$ \sum_{j=0}^\infty\frac{q^{2j(4j-1)}x^{4j}}{(q^6;q^6)_{2j}}
\theta(-q^{4-4j}/x;q^4)
=\theta(-q^{4}/x;q^4)\sum_{j=0}^\infty\frac{q^{6j^2}x^{3j}}{(q^6;q^6)_{2j}}.
$$
In the same way, the second term is
$$
\sum_{j=0}^\infty\frac{q^{(2j+1)(4j+1)}x^{4j+2}}{(q^6;q^6)_{2j+1}}
\theta(-q^{2-4j}/x;q^4)
=qx^2\theta(-q^{2}/x;q^4)
\sum_{j=0}^\infty\frac{q^{6j(j+1)}x^{3j}}{(q^6;q^6)_{2j+1}}. \qedhere
$$
\end{proof}

\begin{proof}[Proof of Theorem \ref{thm-exam5}]
In the special cases $x=1$ and $x=q^2$, \eqref{ex5-reduction} reduces to
$$G(1,q^2)=\frac{(q^4;q^4)_\infty}{(q^3;q^3)_\infty}\left(
2(-q^4;q^4)_\infty^2\sum_{j=0}^\infty\frac{q^{6j^2}}{(q^6;q^6)_{2j}}
+q(-q^{2};q^4)_\infty^2
\sum_{j=0}^\infty\frac{q^{6j(j+1)}}{(q^6;q^6)_{2j+1}}
\right),  $$
$$G(q^2,1)=\frac {(q^4;q^4)_\infty}{(q^3;q^3)_\infty}\left(
(-q^2;q^4)_\infty^2\sum_{j=0}^\infty\frac{q^{6j(j+1)}}{(q^6;q^6)_{2j}}
+2q^5(-q^{4};q^4)_\infty^2
\sum_{j=0}^\infty\frac{q^{6j(j+2)}}{(q^6;q^6)_{2j+1}}
\right). $$
We observe that the four series can be summed by Rogers' identities \eqref{four-rogers}. This leads after simplification to
Theorem \ref{thm-exam5}.
\end{proof}

\section{Some Byproducts and Concluding Remarks}
We give several remarks before closing this paper.

As we have seen in the proof of Theorem \ref{thm-exam-relation}, we may have different double series representations for essentially the same Nahm sum. 
 There is in fact another pair of double sum representations for the
 Nahm sums in Example 5. We can write \eqref{CT-S1} as
 \begin{subequations}\label{S12-new}
\begin{align}
\frac{S_1(q)}{(-q^3;q^3)_\infty}=\mathrm{CT} \frac{\theta(1/z;q)_\infty}{(-z;q^{3/2})_\infty (q^{3/2}z;q^3)_\infty} =\sum_{i,j\geq 0} \frac{(-1)^jq^{\frac{(i+j)^2}{2}-\frac{i}{2}+j}}{(q^{3/2};q^{3/2})_i (q^3;q^3)_j}. \label{S1-new}
\end{align}
In the same way, 
\begin{align}
\frac{S_2(q)}{(-q^3;q^3)_\infty} =\mathrm{CT} \frac{\theta(1/z;q)}{(-q^{1/2}z;q^{3/2})_\infty (q^{1/2}z;q^3)_\infty} 
=\sum_{i,j\geq 0} \frac{(-1)^jq^{\frac{(i+j)^2}{2}}}{(q^{3/2};q^{3/2})_i(q^3;q^3)_j}. \label{S2-new}
\end{align}
\end{subequations}
Note that Theorem \ref{thm-exam5} does not only evaluate the Nahm sums of Example 5 but also the alternative double sums on the right-hand sides of 
\eqref{S12-exp} and \eqref{S12-new}.

In the proof of Theorem \ref{thm-exam-relation}, we started with the double series in Example 5. If we instead start with Example 10, we find an interesting byproduct. Let
\begin{align}
H(x,y):=\sum_{i,j\geq 0} \frac{q^{2i^2+2ij+2j^2}x^iy^j}{(q^3;q^3)_i(q^3;q^3)_j}.
\end{align}
Writing $2i^2+2ij+2j^2=(2i+j)^2/2+3j^2/2$, we find that
\begin{align}
H(x,y)&=\mathrm{CT}\sum_{i=0}^\infty  \frac{q^ix^iz^{2i}}{(q^3;q^3)_i}  \sum_{j=0}^\infty \frac{(-1)^jq^{j(3j+1)/2}y^jz^j}{(q^3;q^3)_j} 
\sum_{k=-\infty}^\infty (-1)^kq^{k(k-1)/2}z^{-k}  \nonumber \\
&=\mathrm{CT}  \frac{(yzq^{2};q^3)_\infty \theta(1/z;q)}{(qxz^2;q^3)_\infty}.
\end{align}
This expression breaks the symmetry between $x$ and $y$. The series $S_3(q)$ can be expressed either as
$H(q^{-2},q^{-1})$ or as $H(q^{-1},q^{-2})$. The first choice leads to the double series in the right-hand side of \eqref{S1-exp},
but with the second choice we obtain a slightly different result. Namely,
\begin{align}
H(q^{-1},q^{-2})&=\mathrm{CT} \frac{(qz;q^3)_\infty \theta(1/z;q)}{ (q^{-1}z^2;q^3)_\infty } 
=\mathrm{CT} \frac{\theta(1/z;q)}{(-qz;q^3)_\infty (q^{-1}z^2;q^6)_\infty} \nonumber \\
&=\mathrm{CT}\sum_{i=0}^\infty \frac{(-1)^iq^iz^i}{(q^3;q^3)_i} \sum_{j=0}^\infty \frac{q^{-j}z^{2j}}{(q^6;q^6)_j} \sum_{k=-\infty}^\infty (-1)^kq^{k(k-1)/2}z^{-k} \nonumber \\
&=\sum_{i,j\geq 0} \frac{q^{\frac{1}{2}i^2+2ij+2j^2+\frac{1}{2}i-2j}}{(q^3;q^3)_i(q^6;q^6)_j}. \label{10-1}
\end{align}
Compared to the series in the right-hand side of \eqref{S1-exp}, the only difference is that $(i,j)$ is replaced by $(-i,-j)$ in the linear term in the exponent of $q$. This leads to the following interesting consequence.
\begin{corollary}\label{cor-id}
We have
\begin{align}\label{equal}
\sum_{i,j\geq 0} \frac{q^{\frac{1}{2}i^2+2ij+2j^2-\frac{1}{2}i+2j}}{(q^3;q^3)_i(q^6;q^6)_j}=\sum_{i,j\geq 0} \frac{q^{\frac{1}{2}i^2+2ij+2j^2+\frac{1}{2}i-2j}}{(q^3;q^3)_i(q^6;q^6)_j}.
\end{align}
\end{corollary}
We end this paper with the natural question: can one prove this corollary directly?

\subsection*{Acknowledgements}
This work was supported by the National Natural Science Foundation of China (12171375) and the Swedish Science Research Council (2020-04221).

\end{document}